\documentclass[14pt]{article}
\usepackage[left=1in,top=1in,right=1in,bottom=1in,letterpaper]{geometry}

\usepackage{amsmath,amssymb,amsthm}
\usepackage{latexsym,amsfonts,amscd,amsxtra,amstext}

\newcommand{\RR}{\mathbb{R}}
\newcommand{\dom}{{\mathrm{dom}}} 
\newcommand{\prox}{{\mathbf{prox}}}

\DeclareMathOperator*{\Min}{minimize}

\newcommand{\st}{\mbox{subject to}}

\newtheorem{theorem}{Theorem}
\newtheorem{definition}{Definition}
\newtheorem{corollary}{Corollary}
\newtheorem{lemma}{Lemma}

\begin{document}

\title{Proximal linearized iteratively reweighted least squares for a class of nonconvex and nonsmooth problems}

\author{
Hui Zhang\thanks{Department of Mathematics and Systems Science,
College of Science, National University of Defense Technology,
Changsha, Hunan, China, 410073. Email: \texttt{hhuuii.zhang@gmail.com}}
\and Tao Sun \thanks{Department of Mathematics and Systems Science,
College of Science, National University of Defense Technology,
Changsha, Hunan, China, 410073. Email: \texttt{nudtsuntao@163.com}}
\and Lizhi Cheng \thanks{The state key laboratory for high performance computation, and Department of Mathematics and Systems Science,
 National University of Defense Technology,
Changsha, Hunan, China, 410073. Email: \texttt{clzcheng@nudt.edu.cn}}}
\date{\today}

\maketitle

\begin{abstract}
For solving a wide class of nonconvex and nonsmooth problems, we propose a proximal linearized iteratively reweighted least squares (PL-IRLS) algorithm. We first approximate the original problem by smoothing methods, and second write the approximated problem into an auxiliary problem by introducing new variables. PL-IRLS is then built on solving the auxiliary problem by utilizing the proximal linearization technique and the iteratively reweighted least squares (IRLS) method, and has remarkable computation advantages. We show that PL-IRLS can be extended to solve more general nonconvex and nonsmooth problems via adjusting generalized parameters, and also to solve nonconvex and nonsmooth problems with two or more blocks of variables. Theoretically, with the help of the  Kurdyka-{\L}ojasiewicz property, we prove that each bounded sequence generated by PL-IRLS globally converges to a critical point of the approximated problem.  To the best of our knowledge, this is the first global convergence result of applying IRLS idea to solve nonconvex and nonsmooth problems. At last, we apply PL-IRLS to solve three representative nonconvex and nonsmooth problems in sparse signal recovery and low-rank matrix recovery and obtain new globally convergent algorithms.
\end{abstract}

\section{Introduction}
In this paper, we consider a broad class of nonconvex and nonsmooth problems with the following form:
\begin{equation*}
(M)\quad \quad \Min F(x):=f(x)+s(x)+\sum_{i=1}^m\|B_ix-c_i\|_2,
\end{equation*}
where $B_i\in \RR^{k_i\times n}, c_i\in\RR^{k_i}, i=1, 2, \cdots, m,$ and the functions $f(x)$ and $s(x)$ satisfy the following properties:

(A) The function $f(x)$ is extended valued (i.e., allowing the inclusion of constraints) and the proximal map of $f(x)$, i.e., the quantity
\begin{equation}
\prox_{c}^f(y):=\arg\min_x\{f(x)+\frac{c}{2}\|x-y\|_2^2\}
\end{equation}
is easy to compute for any given $y\in \RR^n$ and $c>0$. Note that even when $f$ is nonconvex, $\prox_{c}^f(y)$ is also well-defined \cite{bolte2013proximal}. When $f(x)$ is the indicator function $\delta(x, Z)$ defined by
\begin{eqnarray}
\delta(x, Z)=\left\{\begin{array}{ll}
0 &\textrm{if} ~x\in Z,\\
+\infty&\textrm{otherwise},
\end{array} \right.
\end{eqnarray}
the proximal map reduces to the projection operator onto $Z$, defined by
\begin{equation}
P_X(y):=\arg\min \{\|x-y\|_2: x\in Z\}.
\end{equation}

(B) The function $s(x)$ is a differentiable function with a Lipschitz continuous gradient whose Lipschitz continuity modulus is bounded by $L_s$; that is
\begin{equation}
\|\nabla s(u)-\nabla s(v)\|_2\leq L_s\|u-v\|_2,~\textrm{ for all}~~u,v\in\RR^n.
\end{equation}

Throughout the paper, we highlight that no convexity will be assumed in the objective or/and the constraints. In other words, the functions of $f(x)$ and $s(x)$ can be convex and nonconvex. Problem (M) appeared in various applications such as image processing, compressed sensing, low-rank matrix recovery, machine learning, statistics, and more. In many applications, $s(x)$ is usually the differentiable loss function. Both of $f(x)$ and $\sum_{i=1}^m\|B_ix-c_i\|_2$ can be regularization functions modeling different priors known about the desired solution. The former is a directed regularization and the latter needs affine maps. In what follows, we describe a couple of application examples of problem (M).

\textbf{Example 1.} (Sparsity constrained $\ell_1$-norm linear regression) In this application, two types of problems \cite{wright2010dense,cai2013rop,chen2013exact} need to be solved
\begin{itemize}
  \item nonconvex case:
  \begin{subequations}\begin{equation}\label{spar1}
\quad \quad \Min \lambda\|x\|_0 +\|Ax-b\|_1
\end{equation}
or \begin{equation}\label{spar2}
\quad \quad \Min \|Ax-b\|_1, \quad  \st \quad \|x\|_0\leq k,
\end{equation}
\end{subequations}

  \item convex case:  \begin{subequations}
  \begin{equation}\label{spar3}
\quad \quad \Min \lambda\|x\|_1 +\|Ax-b\|_1
\end{equation}
or
\begin{equation}\label{spar4}
\quad \quad \Min \|Ax-b\|_1, \quad  \st \quad \|x\|_1\leq r,
\end{equation}
\end{subequations}
\end{itemize}
where $A\in\RR^{m\times n}, b\in \RR^m$, $\lambda, k, r$ are positive parameters, $\|x\|_1=\sum_{i=1}^n
|x_i|$, and $\|x\|_0$ equals to the number of nonzero entries in $x$.
The problem \eqref{spar1} can be written into the form of (M) with $f(x)=\lambda\|x\|_0$, $s(x)\equiv 0, B_i=e_i^TA, c_i=b_i$, and so is the problem \eqref{spar2} with $f(x)=\delta(x, \Sigma_k)$, $s(x)\equiv 0, B_i=e_i^TA, c_i=b_i$,  where $e_i$ denotes the vector whose $i$th component is 1 and other components are 0, and $\Sigma_k=\{x\in\RR^n: \|x\|_0\leq k\}$. Similarly, the problems \eqref{spar3} and \eqref{spar4} can be also written into the form of (M); we omit the details.

In this group of models, $\|x\|_0$ is used to produce sparse solution, the function $\|Ax-b\|_1$ reflects that the observed data is contaminated by sparse (or say bounded/impulse) noise. In convex case, $\|x\|_1$ is used as a convex relaxation of $\|x\|_0$ for two purposes: not only turning a nonconvex problem into a convex one, but also producing sparse solution. These optimization problems are ubiquitous in compressive sensing community \cite{candes2006stable,donoho2006compressed}.

\textbf{Example 2.} (Cosparse least square problem) In this application, one needs to solve
\begin{equation}\label{cospar}
\Min_{x\in \Omega} \frac{\lambda}{2}\|\Phi x-b\|_2^2+\|\Psi x\|_1,
\end{equation}
 where $\Phi\in\RR^{m_1\times n}, b\in\RR^{m_1}, \Psi\in\RR^{m\times n}$ and $\lambda>0$, $\Omega\subset\RR^n$ is a closed set. Here, $\Psi$ is the so-called analyzing operator in cosparse models \cite{nam2013cosparse}. The problem \eqref{cospar} has the form of (M) with $f(x)=\delta(x,\Omega), s(x)=\frac{\lambda}{2}\|\Phi x-b\|_2^2, B_i=e^T_i\Psi, c_i=0$.

 In \eqref{cospar}, the function $\|\Psi x\|_1$ is used to promote sparsity and can be understood in such a way that the objective/solution is sparse after a transformation. Commonly encountered transformations includes wavelet operator, total variation operator, and redundant frame operator.
 This problem arises from many applications such as the total variation model in image processing \cite{rudin1992nonlinear}, cosparse signal recovery \cite{nam2013cosparse} in compressive sensing and so on.

 \textbf{Example 3.} (Robust principle component analysis, RPCA \cite{candes2011robust}) The purpose of RPCA is decompose an observed matrix $D$ into a sum of a low-rank component and sparse component. Therefore, one may be interested in the following problems
 \begin{itemize}
  \item nonconvex case:
  \begin{subequations}
  \begin{equation}\label{rpca1}
\Min \lambda \cdot\textrm{rank}(X)+\|D-X\|_1
\end{equation}
or   \begin{equation}\label{rpca2}
\Min \|D-X\|_1, \quad  \st \quad \textrm{rank}(X) \leq k,
\end{equation}
\end{subequations}

  \item convex case:   \begin{subequations}
  \begin{equation}\label{rpca3}
\Min \lambda \cdot\|X\|_*+\|D-X\|_1
\end{equation}
or   \begin{equation}\label{rpca4}
\Min \|D-X\|_1, \quad  \st \quad \|X\|_* \leq r,
\end{equation}
\end{subequations}
\end{itemize}
where $D$ the observed matrix, $\textrm{rank}(X)$ is the rank of matrix $X$, $\|X\|_*$ represents the nuclear norm of matrix $X$ and equals to the sum of all singular values of $X$.
All these problems \eqref{rpca1}-\eqref{rpca4} can be viewed as special cases of the general problem (M). For example, problem \eqref{rpca1} has the form of (M) with $f(x)= \lambda \cdot\textrm{rank}(X), s(x)\equiv 0, m=1, B_1=\mathcal{I}, c_1=D$, where $\mathcal{I}$ denotes the identity operator.

 Now, let us return to problem (M). From properties (A) and (B), we know that $f(x)$ is simple in the sense that its proximal map is easy to be computed and $s(x)$ is gradient-Lipschtiz-continuous. The main difficulty in solving problem (M) comes from the last term $\sum_{i=1}^m\|B_ix-c_i\|_2$ which is not smooth. To get around this difficulty, it is natural to smooth this term and solve a smoothed approximation of (M) like
 \begin{equation*}
(M_\epsilon)\quad \quad \Min F_\epsilon(x):=f(x)+s(x)+\sum_{i=1}^m\sqrt{\|B_ix-c_i\|_2^2+\epsilon^2},
\end{equation*}
which was suggested in \cite{beck2013convergence}; other types of smoothing methods can be found in \cite{beck2012smoothing}. When $f(x)=\delta(x,X)$ with $X$ being a closed and convex subset of $\RR^n$, problem ($M_\epsilon$) becomes
 \begin{equation}\label{Beck}
\Min s(x)+\sum_{i=1}^m\sqrt{\|B_ix-c_i\|_2^2+\epsilon^2} \quad \st \quad x\in X,
\end{equation}
which is exactly the problem studied in \cite{beck2013convergence} where the author proposed an iteratively reweighted least square (IRLS) method to solve it. IRLS has a relatively long research history and is a very powerful tool to deal with nonconvex and/or nonsmooth objective functions. Recent works include IRLS for minimizing the $\|x\|_\nu^\nu=\sum_{i=1}^n|x_i|^\nu$ with $0<\nu\leq 1$ in sparse signal recovery \cite{chartrand2008iteratively,candes2008enhancing,daubechies2010iteratively,gasso2009recovering,foucart2009sparsest,lai2011unconstrained,
zhang2012reweighted,zhao2012reweighted,lu2012iterative} and for minimizing the $\|X\|_*$ in low-rank matrix recovery \cite{mohan2010iterative,fornasier2011low,lai2013improved,lu2014iterative}. The connection of IRLS with other well-known algorithms was discovered as well. For example, work \cite{daubechies2010iteratively} pointed out that IRLS is actually the alternating minimization applied to an auxiliary function, and very recent work \cite{ba2013convergence} demonstrated a one-to-one correspondence between the IRLS and a class of Expectation-Maximization (EM) algorithms. By the first connection, the author in \cite{beck2013convergence} established a nonasymptotic sublinear rate of convergence for the IRLS method. In this study, we further develop the IRLS method via the following three-fold contributions:
\begin{enumerate}
  \item We apply the IRLS idea to solve problem ($M_\epsilon$) which is essentially more general than problem \eqref{Beck} since nonconvexity is involved;
  \item We propose a proximal linearized IRLS algorithm solving problem ($M_\epsilon$). In the original IRLS algorithm \cite{beck2013convergence}, the subproblem in each iteration is usually hard to be solved; whilst in our new algorithm, each subproblem has a closed-from formulation for solution due to the simpleness of the proximal map of $f(x)$ and the proximal linearization technique;
  \item We prove that each bounded sequence generated by the proximal linearized IRLS globally converges to a critical point of $F_\epsilon(x)$. To the best of our knowledge, this is the first global convergence result of applying IRLS idea to solve nonconvex and nonsmooth problems. Our method is motivated by the convergence analysis framework in \cite{bolte2013proximal} which is building on the powerful Kurdyka-{\L}ojasiewicz property.
\end{enumerate}

The rest of the paper is organized as follows. In section 2, we list some basic concepts of nonconvex-nonsmooth optimization and introduce the Kurdyka-{\L}ojasiewicz property which is a key tool for global convergence analysis. In section 3, by efficiently exploiting both of the proximal linearization technique and the iteratively reweighted least squares method, we propose the new method--called proximal linearized iteratively reweighted least square (PL-IRLS) algorithm. In section 4, we provide a globally convergence proof for our proposed algorithm by assuming that the objective function $F_\epsilon(x)$ satisfies the Kurdyka-{\L}ojasiewicz property and has a finite lower bound. In section 5, we extend PL-IRLS to solve more general nonconvex-nonsmooth minimization problems than problem ($M_\epsilon$) by adjusting generalized parameters, and also to solve nonconvex-nonsmooth problems with two or more blocks of variables. In section 6, representative application examples of PL-IRLS are given and corresponding algorithms are derived.

\section{Notations and Preliminaries}
\subsection{Basic concepts of nonconvex-nonsmooth optimization}
We collect several definitions as well as some useful properties in optimization from \cite{mordukhovich2006variational}.

For a proper and lower semicontinuous function $\sigma: \RR^n \rightarrow (-\infty, +\infty]$, its domain is defined by
$$\dom (\sigma):=\{x\in\RR^n: \sigma(x)<+\infty\}.$$
The graph of a real-extended-valued function $\sigma: \RR^n \rightarrow (-\infty, +\infty]$ is defined by
$$\textrm{Graph} (\sigma):=\{(x,v)\in\RR^n\times\RR: v=\sigma(x)\}.$$
The notation of subdifferential plays a central role in (non)convex optimization.

\begin{definition}[subdifferentials, \cite{mordukhovich2006variational}] Let  $\sigma: \RR^n \rightarrow (-\infty, +\infty]$ be a proper and lower semicontinuous function.
\begin{enumerate}
  \item For a given $x\in \dom (\sigma)$, the Fr$\acute{e}$chet subdifferential of $\sigma$ at $x$, written $\hat{\partial}\sigma(x)$, is the set of all vectors $u\in \RR^n$ which satisfy
  $$\lim_{y\neq x}\inf_{y\rightarrow x}\frac{\sigma(y)-\sigma(x)-\langle u, y-x\rangle}{\|y-x\|}\geq 0.$$
When $x\notin \dom (\sigma)$, we set $\hat{\partial}\sigma(x)=\emptyset$.

\item The limiting-subdifferential, or simply the subdifferential, of $\sigma$ at $x\in \RR^n$, written $\partial\sigma(x)$, is defined through the following closure process
$$\partial\sigma(x):=\{u\in\RR^n: \exists x^k\rightarrow x, \sigma(x^k)\rightarrow \sigma(x)~\textrm{and}~ u^k\in \hat{\partial}\sigma(x^k)\rightarrow u~\textrm{as}~k\rightarrow \infty\}.$$
\end{enumerate}
\end{definition}
We will need the closed-ness property of $\partial \sigma(x)$:

Let $\{(x^k,v^k)\}_{k\in \mathbb{N}}$ be a sequence in $\RR^n\times \RR$ such that $(x^k,v^k)\in \textrm{Graph }(\partial \sigma)$. If $(x^k,v^k)$  converges to $(x, v)$ as $k\rightarrow +\infty$ and $\sigma(x^k)$ converges to $\sigma(v)$ as $k\rightarrow +\infty$, then $(x, v)\in \textrm{Graph }(\partial \sigma)$.

A necessary condition for $x\in\RR^n$ to be a minimizer of $\sigma(x)$ is
\begin{equation}\label{Fermat}
0\in \partial \sigma(x).
\end{equation}
A point that satisfies \eqref{Fermat} is called (limiting-) critical point. The set of critical points of $\sigma(x)$ is denoted by $\textrm{crit}(\sigma)$.

\subsection{The Kurdyka-{\L}ojasiewicz property}
Let $\sigma: \RR^n \rightarrow (-\infty, +\infty]$ be a proper and lower semicontinuous function. For given real numbers $\alpha$ and $\beta$, we set
$$[\alpha<\sigma<\beta]:=\{x\in \RR^n: \alpha <\sigma(x)<\beta\}.$$
 For any subset $S\subset \RR^n$ and any point $x\in\RR^n$, the distance from $x$ to $S$ is defined by
$$\textrm{dist}(x, S):=\inf \{\|x-y\|:y\in S\}.$$
We take the following definition of the Kurdyka-{\L}ojasiewicz property from \cite{attouch2010proximal,bolte2010characterizations}

\begin{definition}[ Kurdyka-{\L}ojasiewicz property and function]
(a) The function $\sigma: \RR^n \rightarrow (-\infty, +\infty]$ is said to have the  Kurdyka-{\L}ojasiewicz property at $x^*\in\dom(\partial \sigma)$ if there exist $\eta\in (0, +\infty]$, a neighborhood $U$ of $x^*$ and a continuous function $\varphi: [0, \eta)\rightarrow\RR_+$ such that
\begin{enumerate}
  \item $\varphi(0)=0$.
  \item $\varphi$ is $C^1$ on $(0, \eta)$.
  \item for all $s\in(0, \eta)$, $\varphi^{'}(s)>0$.
  \item for all $x$ in $U\bigcap[\sigma(x^*)<\sigma<\sigma(x^*)+\eta]$, the Kurdyka-{\L}ojasiewicz inequality holds
  \begin{equation}
  \varphi^{'}(\sigma(x)-\sigma(x^*))\textrm{dist}(0,\partial \sigma(x))\geq 1.
\end{equation}
\end{enumerate}

(b) Proper lower semicontinuous functions which satisfy the Kurdyka-{\L}ojasiewicz inequality at each point of $\dom(\partial \sigma)$ are called KL functions.
\end{definition}

The Kurdyka-{\L}ojasiewicz (KL) inequality was originally created in \cite{lojasiewicz1963propriete} and\cite{kurdyka1998gradients}. Then, extensions to nonsmooth cases were made in \cite{bolte2007clarke,bolte2007lojasiewicz,bolte2010characterizations}. The concept of semi-algebraic sets and functions can help find and check a very rich class of Kurdyka-{\L}ojasiewicz functions.

\begin{definition}[Semi-algebraic sets and functions, \cite{attouch2013convergence}]
(i) A subset $S$ of $\RR^n$ is a real semi-algebraic set if there exists a finite number of real polynomial functions $g_{ij}, h_{ij}:\RR^n\rightarrow \RR$ such that
$$S=\bigcup_{j=1}^p\bigcap_{i=1}^q\{u\in\RR^n:g_{ij}(u)=0~\textrm{and}~~h_{ij}(u)<0\}.$$

(ii) A function $h:\RR^n\rightarrow (-\infty, +\infty]$ is called semi-algebraic if its graph
$$\{(u, t)\in \RR^{n+1}: h(u)=t\}$$
is a semi-algebraic subset of $\RR^{n+1}$.
\end{definition}

\begin{lemma}[Semi-algebraic property implies KL property, \cite{bolte2007lojasiewicz,bolte2007clarke}]\label{semialge}
Let $\sigma ¦Ò:\RR^n\rightarrow \RR$ be a proper and lower semicontinuous function. If $\sigma$ is semi-algebraic then it satisfies the KL property at any point of $\textrm{dom} (\sigma)$. In particular, If $\sigma$ is semi-algebraic and $\textrm{dom} (\sigma)=\dom(\partial \sigma)$, then it is a KL function.
\end{lemma}

The authors in \cite{bolte2013proximal} based on the lemma above listed a broad class of semi-algebraic function (or KL functions) in optimization. Examples include finite sums of semi-algebraic (KL) functions, composition of semi-algebraic (KL) functions, and so on.

Recently, the  Kurdyka-{\L}ojasiewicz inequality has become an important and even standard tool for convergence analysis of iterative algorithms for nonconvex-nonsmooth minimization problems \cite{attouch2009convergence,attouch2010proximal,attouch2013convergence,bolte2013proximal}. In this paper, we will require the following result about the KL inequality to show the global convergence of PL-IRLS.

\begin{lemma}
 [Uniformized KL property, \cite{bolte2013proximal}] \label{UKL} Let $\Omega$ be a compact set and let $\sigma ¦Ò:\RR^n\rightarrow \RR$ be a proper and lower semicontinuous function.  Assume that $\sigma$ is constant on $\Omega$ and satisfies the KL property at each point of $\Omega$. Then, there exist $\zeta>0, \eta>0$ and $\varphi\in \Phi_\eta$ such that for all $\bar{u}$ and all $u$ in the following intersection:
 \begin{equation}
 \{u\in\RR^n:\textrm{dist}(u,\Omega)<\zeta\}\bigcap [\sigma(\bar{u})<\sigma(u)<\sigma(\bar{u})+\eta],
\end{equation}
one has,
\begin{equation}
 \varphi^{'}(\sigma(u)-\sigma(\bar{u}))\textrm{dist}(0,\partial \sigma(u))\geq 1.
\end{equation}
\end{lemma}

\section{The proposed algorithm}
We start by introducing an auxiliary problem
\begin{equation*}
(AM)\quad\quad \Min \Psi(x, y)=f(x)+H(x,y)+g(y),
\end{equation*}
where $f(x)$ and $g(y)$ are extended valued and $H(x,y)$ is a smooth function. This type of problems have been studied in several recent papers \cite{attouch2010proximal,xu2012block,bolte2013proximal}. Here, we focus on a special form: $H(x,y)=s(x)+\sum_{i=1}^m(\|B_ix-c_i\|_2^2+\epsilon^2)y_i$, $f(x)$ is the same function as that in (M), and $g(y)=\sum_{i=1}^m\frac{1}{4y_i}+\delta(y, \Lambda)$ with $\Lambda=(0, \frac{\epsilon}{2}]^m$. In other words, we actually try to solve the following problem
\begin{equation*}
(AMs)\quad\quad \Min \Psi(x, y)=f(x)+\underbrace{s(x)+\sum_{i=1}^m(\|B_ix-c_i\|_2^2+\epsilon^2)y_i}_{H(x,y)}+\underbrace{\sum_{i=1}^m\frac{1}{4y_i}+\delta(y, \Lambda)}_{g(y)}.
\end{equation*}

Comparing the auxiliary objective above with that in ($M_\epsilon$), we can find that the objective with respect to $x$ becomes much nicer after introducing the auxiliary variable $y$. More importantly, this auxiliary problem is equivalent to the smoothed approximation problem ($M_\epsilon$) in the sense that they enjoy the same minimizer set of $x$ variable. Indeed, we have
\begin{lemma}
Assume that $-\infty<\min F_\epsilon(x)$. Let $(X^*, Y^*)=\arg\min_{(x,y)} \Psi(x, y), \hat{X}=\arg\min_x F_\epsilon(x)$. Then, $X^*=\hat{X}$.
\end{lemma}
\begin{proof}
First, we have
$\min_{x,y}\Psi(x,y)=\min_x\min_y\Psi(x,y)=\min_x F_\epsilon(x)\triangleq \bar{F}$. On one hand, for any $x^*\in \hat{X}$, take $y^*=\frac{1}{\sqrt{\|B_ix^*-c_i\|_2^2+\epsilon^2}}$; then it is easy to check that $\Psi(x^*, y^*)=F_\epsilon(x^*)=\bar{F}$ which implies $x^*\in X^*$ and hence $\hat{X}\subset X^*$. On the other hand, for any $(\hat{x},\hat{y})\in (X^*, Y^*)$, letting $\bar{y}=\frac{1}{\sqrt{\|B_i\hat{x}-c_i\|_2^2+\epsilon^2}}$, since $\bar{y}\in\arg\min \Psi(\hat{x}, y)$, we have $\bar{F}\leq F_\epsilon(\hat{x})=\Psi(\hat{x},\bar{y})\leq  \Psi(\hat{x},\hat{y})= \bar{F}$ which implies $\hat{x}\in \hat{X}$ and hence $X^*\subset\hat{X}$. Therefore, $X^*=\hat{X}$.
\end{proof}

There are many methods to solve problem (AM). The primal idea should be applying the alternating minimization method to (AM) to yield the following scheme:
\begin{subequations}
\begin{equation}
x^{k+1}\in\arg\min_x\Psi(x,y^k)
\end{equation}
\begin{equation}
y^{k+1}\in\arg\min_y\Psi(x^{k+1},y).
\end{equation}
\end{subequations}
Replacing the especial expression of $\Psi(x,y)$ of (AMs) into the scheme above and after some simple calculations, we obtain
\begin{equation}
x^{k+1}\in\arg\min_x f(x)+s(x)+\sum_{i=1}^m\frac{\|B_ix-c_i\|_2^2}{2\sqrt{\|B_ix^k-c_i\|_2^2+\epsilon^2}},
\end{equation}
which is exactly the IRLS method proposed in \cite{beck2013convergence} given that $f(x)=\delta(x, X)$. As mentioned before, the main difficulty is to solve the subproblem in each iteration. Additionally, the nonconvex function $f(x)$ entering into the objective also makes the computation and convergence analysis become harder.

Very recently, the authors in \cite{bolte2013proximal} proposed a rather powerful algorithm, namely proximal alternating linearized minimization (PALM) algorithm, to solve a wide class of nonconvex-nonsmooth problems of the form (AM). The PALM overcomes some drawbacks of the alternating minimization method and has global convergence property if the objective function $\Psi(x,y)$ is a KL function and some assumptions are met. Recall that in the alternating minimization we need to minimize $\Psi(x,y^k)=f(x)+H(x,y^k)$ and $\Psi(x^{k+1},y)=g(y)+H(x^{k+1},y)$ both of which are the sum of a smooth function with a nonsmooth one. The main idea of PALM is proximally linearizing the smooth function and keeping the nonsmooth function. Concretely, PALM algorithm reads
\begin{subequations}
\begin{equation}
x^{k+1}\in\arg\min_x f(x)+\langle x-x^k,\nabla_xH(x^k,y^k)\rangle+\frac{c_k}{2}\|x-x^k\|_2^2
\end{equation}
\begin{equation}
y^{k+1}\in\arg\min_y g(y)+\langle y-y^k,\nabla_yH(x^{k+1},y^k)\rangle+\frac{d_k}{2}\|y-y^k\|_2^2,
\end{equation}
\end{subequations}
where $c_k>0, d_k>0$ are step parameters. Using the proximal map notation, PALM can be equivalently written as
\begin{subequations}
\begin{equation}\label{xk}
x^{k+1}\in\prox_{c_k}^f(x^k-\frac{1}{c_k}\nabla_xH(x^k,y^k))
\end{equation}
\begin{equation}
y^{k+1}\in\prox_{d_k}^g(y^k-\frac{1}{d_k}\nabla_yH(x^{k+1},y^k)).
\end{equation}
\end{subequations}
Although the PALM algorithm enjoys many nice properties such as each step is relatively easy to be computed and $(x^k,y^k)$ globally converges to a critical point of $\Psi(x,y)$, it may not fit our problem very well. We list some reasons here. First, in our case $H(x^{k+1},y)$ is a linear function, so itself is quite simple and does not need to be linearized; Second, without linearizing, we can directly minimize $\Psi(x^{k+1},y)=g(y)+H(x^{k+1},y)$ and get
\begin{equation}\label{yk}
y_i^{k+1}=\frac{1}{2\sqrt{\|B_ix^{k+1}-c_i\|_2^2+\epsilon^2}},~i=1, 2, \cdots, m.
\end{equation}
On the contrary, minimizing the sum of $g(y)$ and the proximal linearization of $H(x^{k+1},y)$ is equivalent to minimizing a cubical function which is harder than minimizing $\Psi(x^{k+1},y)=g(y)+H(x^{k+1},y)$. Last but not least, although $(x^k,y^k)$ generated by PALM globally converges to a critical point of $\Psi(x,y)$, what we actually need is to generate a sequence  $\{x^k\}$ that converges to a critical point of $F_\epsilon(x)$. Based on these considerations, we propose the following algorithm (PL-IRLS):

\begin{enumerate}
  \item[1] Initialization: start with any $(x^0, y^0)\in\RR^n\times \RR^m$.
  \item[2] For each $k=0, 1, \cdots$ generate a sequence $\{(x^k,y^k)\}_{k\in \mathbb{N}}$ as follows:
  \begin{enumerate}
    \item[(a)] Take $\gamma>1$, set $c_k=\gamma L(\tau, y^k)$ where $L(\tau, y)$ will be given in Corollary \ref{cor1} and compute $x^{k+1}$ by utilizing \eqref{xk}.

    \item[(b)] Compute $y^{k+1}$ by utilizing \eqref{yk}.
  \end{enumerate}
\end{enumerate}
\textbf{Remark:} Our algorithm can also be derived from the proximal forward-backward (PFB) scheme in \cite{bolte2013proximal}. In fact, letting
\begin{equation}\label{gfunc}
h(x)=s(x)+\sum_{i=1}^m \sqrt{\parallel B_{i}x-c_{i}\parallel^{2}+\varepsilon^{2}}
\end{equation}
and applying PFB to ($M_{\varepsilon}$) yield
\begin{eqnarray}
  x^{k+1}&\in& \arg\min_{x\in R^{n}}\left(\langle x-x^{k},\nabla h(x^{k})\rangle+\frac{c_{k}}{2}\parallel x-x^{k}\parallel^{2}+f(x)\right) \nonumber\\
  &&=\prox_{c_{k}}^{f}(x^{k}-\frac{1}{c_{k}}\nabla h(x^{k}))=\prox_{c_{k}}^{f}(x^{k}-\frac{1}{c_{k}}\nabla_{x} H(x^{k},y^{k})),\nonumber
\end{eqnarray}
where $y^{k}, k=0,1, 2, \cdots$ are given in (19). This is exactly \eqref{xk}. However, proposition 3 in \cite{bolte2013proximal} can not be directly used to guarantee a global convergence of $\{x^k\}$ because that $\nabla h$ fails to be globally Lipschitz continuous (see Lemma \ref{pLip}). Besides, the idea of IRLS can not be well reflected by the PFB scheme. Most importantly, using the idea of IRLS, PL-IRLS can be easily extended to solve problems with two or more blocks of variables (see Section 5) while the PFB scheme seems limited to the problem with one block of variables.

The next section is devoted to analyze PL-IRLS.

\section{Convergence analysis}
The aim in this part is at proving that $\{x^k\}$ generated by the PL-IRLS algorithm globally converges to a critical point of $F_\epsilon(x)$. Our proof is motivated by the general methodology in \cite{bolte2013proximal} and consists of three main steps:
\begin{enumerate}
  \item Sufficient decrease property: Find a positive constant $\rho_1$ such that
  $$ \rho_1 \|x^{k+1}-x^k\|_2^2\leq F_\epsilon(x^k)-F_\epsilon(x^{k+1}), \forall k=0,1, \cdots.$$
  \item A subgradient lower bound for the iterates gap:
        Find another positive constant $\rho_2$  such that
        $$\|w^{k+1}\|_2\leq \rho_2\|x^{k+1}-x^k\|_2,  ~w^k\in \partial F_\epsilon(x^k), ~\forall k=0,1, \cdots.$$
  \item Using the KL property: Assume that $F_\epsilon(x)$ is a KL function and show that the generated sequence $\{x^k\}$ is a Cauchy sequence.
\end{enumerate}
Different from that general methodology in \cite{bolte2013proximal}, our line of thought begins with an assumption that the sequence generated by the algorithm PL-IRLS is bounded, and then sufficiently utilizes the locally gradient-Lipschitz-continuous property given in Definition \ref{localLip}. The advantages of our method include that we do not need to make the additional assumptions as these in \cite{bolte2013proximal} on the coupled function $H(x,y)$, and that we do not need the globally gradient-Lipschitz-continuous property of $h(x)$ given in \eqref{gfunc}.  In the following, we highlight our theoretical contributions:
\begin{enumerate}
  \item[(a)] We define the Locally gradient-Lipschitz property. Together with the assumption that the sequence generated by the algorithm PL-IRLS is bounded, we obtain a global convergence result. Our convergence theory indicates that the assumption of globally gradient-Lipschitz property in \cite{bolte2013proximal} can be weaken into local version. This hence expands the range of the general theory framework in \cite{bolte2013proximal}.
  \item[(b)] We derive detailed parameters which could be used to help us choose the step parameters $c_k, k=0, 1,\cdots,$ in PL-IRLS. The calculation of parameters involved is one of the main differences between our proof and the proof in \cite{bolte2013proximal}, although the outline is the same.
\end{enumerate}

\subsection{Objective function properties}
First, we need to define the following locally gradient-Lipschitz property because we will assume that the sequence generated by the algorithm PL-IRLS is bounded.

\begin{definition}[Locally gradient-Lipschitz property]\label{localLip} Let $h:\RR^n\rightarrow \RR$ be a continuously differentiable function. It is called $L_h^\tau$-locally-gradient-Lipschitz on $\mathcal{B}(\tau):=\{x\in\RR^n: \|x\|_2\leq\tau\}$ if the following holds
$$\|\nabla h(u)-\nabla h(v)\|_2\leq L^\tau_h\|u-v\|_2, ~~\forall u, v\in \mathcal{B}(\tau),$$
where $L^\tau_h$ is a positive constant depending on the parameter $\tau$.
\end{definition}

The locally gradient-Lipschitz property implies decrease properties of objective functions:

\begin{lemma}[Decrease property of single function]\label{hfunLip}
 Let $h:\RR^n\rightarrow \RR$ be $L_h^\tau$-locally-gradient-Lipschitz on $\mathcal{B}(\tau)$. Then, for all $u, v\in \mathcal{B}(\tau)$ we have
 $$h(u)\leq h(v)+\langle \nabla h(v), u-v\rangle +\frac{L^\tau_h}{2}\|u-v\|_2^2,~~\forall u, v\in \mathcal{B}(\tau).$$
\end{lemma}
\begin{proof}
For all $u, v\in \mathcal{B}(\tau)$, we derive that
\begin{align}
h(u) &=h(v)+\int_0^1 \langle \nabla h(v+t(u-v)), u-v\rangle dt \nonumber \\
&=h(v)+\langle \nabla h(v), u-v\rangle +\int_0^1 \langle \nabla h(v+t(u-v))-\nabla h(v), u-v\rangle dt \nonumber \\
&\leq h(v)+\langle \nabla h(v), u-v\rangle + \int_0^1 \|\nabla h(v+t(u-v))-\nabla h(v)\|_2\| u-v\|_2 dt, \label{int}
\end{align}
where the inequality above follows from the Cauchy-Schwartz inequality. Since $u, v\in \mathcal{B}(\tau)$, their convex combination $v+t(u-v)$ must lie in  $\mathcal{B}(\tau)$ so that we can use the locally gradient-Lipschitz property of $h(x)$ to get that
\begin{equation}\label{inint}
\|\nabla h(v+t(u-v))-\nabla h(v)\|_2\leq t\cdot L^\tau_h\|u-v\|_2, ~~t\in[0, 1].
\end{equation}
Thus, combining \eqref{int} and \eqref{inint} yields the final assertion.
\end{proof}

\begin{lemma}
[Sufficient decrease property of sum functions]\label{PSD} Let $h:\RR^n\rightarrow \RR$ be $L_h^\tau$-locally-gradient-Lipschitz on $\mathcal{B}(\tau)$ and let $\sigma ¦Ò:\RR^n\rightarrow \RR$ be a proper and lower semicontinuous function with $\inf \sigma>-\infty$. Fix
any $t>L^\tau_h$. Let $u^+\in \prox_t^\sigma(u-\frac{1}{t}\nabla h(u))$ and assume that both $u$ and $u^+$ lie in $\mathcal{B}(\tau)$. Then we have
$$h(u^+)+\sigma(u^+)\leq h(u)+\sigma(u)-\frac{1}{2}(t-L^\tau_h)\|u^+-u\|_2^2.$$
\end{lemma}
\begin{proof}
On one hand, by the definition of the proximal map, we rewrite $u^+\in \prox_t^\sigma(u-\frac{1}{t}\nabla h(u))$ as follows
\begin{align}
u^+ &\in\arg\min_x\left(\sigma(x)+\frac{t}{2}\|x-u+\frac{1}{t}\nabla h(u)\|_2^2\right) \nonumber \\
&=\arg\min_x\left(\underbrace{\sigma(x)+\langle x-u,\nabla h(u)\rangle + \frac{t}{2}\|x-u\|_2^2}_{G(x)}\right) \nonumber.
\end{align}
Since $u^+$ minimizes $G(x)$, it holds that $G(u^+)\leq G(u)$ or
\begin{equation}\label{uone}
\sigma(u^+)+\langle u^+-u,\nabla h(u)\rangle + \frac{t}{2}\|u^+-u\|_2^2\leq \sigma(u).
\end{equation}
On the other hand, by Lemma \ref{hfunLip} and the assumption that $u$ and $u^+$ lie in $\mathcal{B}(\tau)$, we have
\begin{equation}\label{utwo}
h(u^+)\leq h(u)+\langle u^+-u,\nabla h(u)\rangle + \frac{L^\tau_h}{2}\|u^+-u\|_2^2.
\end{equation}
Thus, summing up \eqref{uone} and \eqref{utwo} yields the conclusion.
\end{proof}

In order to study the locally gradient-Lipschitz property of $H(x,y)$ in problem (AMs), we define that
\begin{equation}\label{pi}
p_i(x)=\sqrt{\|B_ix-c_i\|_2^2+\epsilon^2}, i=1, 2, \cdots.
\end{equation}

\begin{lemma}\label{pLip}
 Let $p_i(x)$ be defined in \eqref{pi}. Then, we have
 $$ \|\nabla p_i(u)-\nabla p_i(v)\|_2\leq L^\tau_{i}\|u-v\|_2, ~~\forall u, v\in \mathcal{B}(\tau),$$
 where $\tau$ is a positive constant and
 $$L_i^\tau=\frac{\|B_i\|\|c_i\|_2+\|B^T_iB_i\|(2\tau\|B_i\|+ \|c_i\|_2+\epsilon)}{\epsilon^2}.$$
 In particular, $L_i^\tau\rightarrow +\infty$ as $\tau\rightarrow +\infty$.
\end{lemma}
\begin{proof}
First, we write down the gradient of $p_i(x)$ as follows
$$\nabla p_i(x)=\frac{B_i^T(B_ix-c_i)}{\sqrt{\|B_ix-c_i\|_2^2+\epsilon^2}}=\frac{B_i^T(B_ix-c_i)}{p_i(x)},~ i=1,2\cdots, m.$$
It is easy to see that $\|\nabla p_i(x)\|_2\leq \|B_i\|$.

Second, we show that for all $u, v\in\RR^n$, the following inequalities hold
\begin{equation}\label{hLip}
|p_i(u)-p_i(v)|\leq \|B_i\|\|u-v\|_2,~ i=1,2\cdots, m.
\end{equation}
Indeed, let $q_i(t)=p_i(v+t(u-v)), 0\leq t\leq 1, i=1,2\cdots, m$; then by the mean-value theorem and the Cauchy-Schwartz inequality, we derive that
\begin{align}
|p_i(u)-p_i(v)| &=|q_i(1)-q_i(0)|= |q_i^{'}(\xi_i)|, ~\textrm{for some}~~~\xi_i\in [0, 1], \nonumber \\
&=|\langle \nabla p_i(v+\xi_i(u-v)), u-v\rangle| \nonumber \\
&\leq  \|\nabla p_i(v+\xi_i(u-v))\|_2\|u-v\|_2 \nonumber \\
& \leq  \|B_i\|\|u-v\|_2.
\end{align}

At last, we give a bound of $ \|\nabla p_i(u)-\nabla p_i(v)\|_2$ via the following deriving
\begin{align}
\|\nabla p_i(u)-\nabla p_i(v)\|_2 &=\|B_i^TB_i(\frac{u}{p_i(u)}-\frac{v}{p_i(v)})+c_i(\frac{1}{p_i(v)}-\frac{1}{p_i(u)})\|_2, \nonumber \\
&=\|\frac{B_i^TB_i}{p_i(u)p_i(v)}(up_i(v)-up_i(u)+up_i(u)-vp_i(u))+\frac{c_i}{p_i(v)p_i(u)}(p_i(u)-p_i(v))\|_2 \nonumber \\
&\leq \frac{\|B_i^TB_i\|}{\epsilon^2}( \|u\|_2|p_i(u)-p_i(v)| +|p_i(u)|\|u-v\|_2)+\frac{\|c_i\|}{\epsilon^2}|p_i(u)-p_i(v)|\label{bound}.
\end{align}
where we have used that $p_i(u)p_i(v)\geq \epsilon^2$. For $u\in\mathcal{B}(\tau)$, it holds
\begin{equation}\label{pval}
|p_i(u)|=\sqrt{\|B_iu-c_i\|_2^2+\epsilon^2}\leq \|B_iu-c_i\|_2+\epsilon\leq \|B_i\|\tau+\|c_i\|_2+\epsilon.
\end{equation}
The desired bound follows by using \eqref{hLip} and \eqref{pval} to \eqref{bound}.
\end{proof}

With the notation of $p_i(x)$, the function $H(x,y)$ in problem (AMs) can be written as
$H(x, y)=s(x)+\sum_{i=1}^m p_i(x)y_i$.

\begin{corollary}\label{cor1}
Denote $L^\tau_p=(L_1^\tau, L_2^\tau, \cdots, L_m^\tau )^T$. For any fixed $y\in\RR^m$, the function $H(x,y)$ with respect to variable $x$ is $L(\tau, y)$-locally-gradient-Lipschitz on $\mathcal{B}(\tau)$ with
$L(\tau, y)=L_s+\|L^\tau_p\|_1\|y\|_\infty$.
\end{corollary}
\begin{proof}
Applying property (B) and Lemma \ref{pLip}, for all $u, v\in\RR^n$ we derive that
\begin{align}
\|\nabla_xH(u,y)-\nabla_xH(v,y)\|_2 &=\|\nabla s(u)-\nabla s(v)+\sum_{i=1}^m(\nabla p_i(u)-\nabla p_i(v))y_i\|_2, \nonumber \\
&\leq \|\nabla s(u)-\nabla s(v)\|_2+\sum_{i=1}^m\|\nabla p_i(u)-\nabla p_i(v))\||y_i| \nonumber \\
&\leq  L_s\|u-v\|_2 + \sum_{i=1}^mL_i^\tau\|u-v\|_2|y_i| \nonumber \\
& \leq  (L_s+ \sum_{i=1}^mL_i^\tau|y_i|)\|u-v\|_2 \nonumber.
\end{align}
By the Cauchy-Schwartz inequality, we get $\sum_{i=1}^mL_i^\tau|y_i|\leq \|L^\tau_p\|_1\|y\|_\infty$ where $\|L^\tau_p\|_1=\sum_{i=1}^m|L^\tau_i|$ and $\|y\|_\infty=\max_i |y_i|$. Therefore, $$\|\nabla_xH(u,y)-\nabla_xH(v,y)\|_2\leq (L_s+ \|L_p\|_1\|y\|_\infty)\|u-v\|_2$$ which completes the proof.
\end{proof}

\subsection{Iteration sequences and limit points}
Before stating the main theorem, we need to prove two lemmas below. In the first one, we establish basic convergence properties of the iteration sequence generated by PL-IRLS.
\begin{lemma}[Basic convergence properties]\label{PIS}
Let $\{x^k\}_{k\in\mathbb{N}}$ be a sequence generated by PL-IRLS and assume that $\inf F_\epsilon>-\infty$ and there exists a constant $\tau$ big enough such that $x^k\in \mathcal{B}(\tau), k=0, 1, \cdots.$ Denote
$$w^k:=\nabla_xH(x^k,y^k)-\nabla_xH(x^{k-1},y^k)+c_{k-1}(x^{k-1}-x^k).$$
Then, the followings hold
\begin{enumerate}
  \item[(i)] The sequence $\{F_\epsilon(x^k)\}_{k\in\mathbb{N}}$  is nonincreasing and in particular
  \begin{equation}\label{seqp1}
 \rho_1 \|x^{k+1}-x^k\|_2^2\leq F_\epsilon(x^k)-F_\epsilon(x^{k+1}), \forall k=0,1, \cdots,
\end{equation}
where $\rho_1=\frac{(\gamma-1)L_s}{2}.$
  \item[(ii)] We have  \begin{equation}
 \sum_{i=1}^\infty \|x^{k+1}-x^k\|^2_2\leq \infty
\end{equation}
and hence $\lim_{k\rightarrow  \infty}(x^{k+1}-x^k)=0.$

\item[(iii)] $w^k\in \partial F_\epsilon(x^k)$ and $\|w^k\|_2\leq \rho_2\|x^k-x^{k-1}\|_2, \forall k\geq 0$ where $\rho_2=(\gamma+1)L_s+(\frac{\gamma}{2\epsilon}+1)\|L_p^\tau\|_1$.
\end{enumerate}
\end{lemma}
\begin{proof}
\textit{(i)}  Since $H(\cdot, y^k)$ is $L(\tau, y^k)$-locally-gradient-Lipschitz on $\mathcal{B}(\tau)$ from Corollary \ref{cor1}, applying Lemma \ref{PSD} with $h(\cdot)=H(\cdot, y^k), \sigma(x)=f(x), t=c_k>L(\tau, y^k)$ and using the first iterative step \eqref{xk} in PL-IRLS, we obtain that
\begin{align}
 H(x^{k+1}, y^k) +f(x^{k+1})&\leq H(x^k, y^k) +f(x^k)-\frac{1}{2}(c_k-L(\tau, y^k))\|x^{k+1}-x^k\|_2^2\nonumber\\
  &= H(x^k, y^k) +f(x^k)-\frac{\gamma-1}{2}(L_s+\|L_p^\tau\|_1\|y^k\|_\infty)\|x^{k+1}-x^k\|_2^2\label{iter1r}.
\end{align}
From the second iterative step \eqref{yk}, we get that
\begin{equation}\label{iter2r}
 H(x^{k+1}, y^{k+1}) +g(x^{k+1})\leq H(x^{k+1}, y^k) +g(x^k),
\end{equation}
and $\|y^k\|_\infty \leq \frac{1}{2\epsilon}$. By \eqref{iter1r} and \eqref{iter2r}, we thus get that for all $k\geq 0$,
\begin{align}\label{iter12r}
\Psi(x^k, y^k)-\Psi(x^{k+1}, y^{k+1})& = H(x^k, y^k) +f(x^k) -H(x^{k+1}, y^{k+1}) -g(x^{k+1}) \nonumber\\
&\geq \frac{\gamma-1}{2}(L_s+\|L_p^\tau\|_1\|y^k\|_\infty)\|x^{k+1}-x^k\|_2^2 \nonumber\\
& \geq \frac{(\gamma-1)L_s}{2}\|x^{k+1}-x^k\|_2^2=\rho_1 \|x^{k+1}-x^k\|_2^2.
\end{align}
It remains to show that $\Psi(x^k, y^k)=F_\epsilon(x^k)$ for all $k\geq 0$. Indeed, we have that
\begin{align}
\Psi(x^k, y^k)&=f(x^k)+H(x^k,y^k)+g(y^k) \nonumber\\
& =f(x^k)+s(x^k)+\sum_{i=1}^m(\|B_ix^k-c_i\|_2^2+\epsilon^2)y_i^k+\sum_{i=1}^m\frac{1}{4y_i^k} \nonumber\\
& =f(x^k)+s(x^k)+\sum_{i=1}^m\left[\frac{\|B_ix^k-c_i\|_2^2+\epsilon^2}{2\sqrt{\|B_ix^k-c_i\|_2^2+\epsilon^2}}
+\frac{1}{2}\sqrt{\|B_ix^k-c_i\|_2^2+\epsilon^2}\right]
\nonumber\\
& =f(x^k)+s(x^k)+\sum_{i=1}^m \sqrt{\|B_ix^k-c_i\|_2^2+\epsilon^2}=F_\epsilon(x^k),
\end{align}
where we implicitly used the fact that $\delta(y^k, \Lambda)=0$ since $\|y^k\|_\infty \leq \frac{1}{2\epsilon}$ and $y^k_i>0$.

\textit{(ii)} Summing up \eqref{seqp1} from $k=0$ to $N-1$, we obtain that
$$ \sum_{i=1}^N \|x^{k+1}-x^k\|^2_2\leq \rho_1(F_\epsilon(x^0)-F_\epsilon(x^N))\leq \rho_1(F_\epsilon(x^0)-\inf F_\epsilon).$$
Taking $N\rightarrow \infty$ ,we get the desired assertion.

\textit{(iii)}On one hand, by \eqref{xk} and the definition of proximal map, we have that
\begin{equation}
x^{k}\in\arg\min_x\left( f(x)+\langle x-x^{k-1},\nabla_xH(x^{k-1}, y^{k-1})\rangle + \frac{c_{k-1}}{2}\|x-x^{k-1}\|_2^2\right).
\end{equation}
Writing down the optimality conditions yields
\begin{equation}
 u^k=c_{k-1}(x^{k-1}-c^{k-1})-\nabla_xH(x^{k-1}, y^{k-1}),
\end{equation}
where $u^k\in f(x^k)$.  On the other hand, it is clear to see that
\begin{align}
\nabla_xH(x^k,y^k)&=\nabla s(x^k)+\sum_{i=1}^m2B^T_i(B_ix^k-c_i)y^k_i \nonumber\\
& =\nabla s(x^k)+\sum_{i=1}^m \frac{B^T_i(B_ix^k-c_i)}{\sqrt{\|B_ix^k-c_i\|_2^2+\epsilon^2}} \nonumber\\
&=\nabla s(x^k)+\sum_{i=1}^m \nabla p_i(x^k),
\end{align}
which implies that $\nabla_xH(x^k, y^k)+u^k\in \partial F_\epsilon(x^k)$. Thus,
$$w^k=\nabla_xH(x^k,y^k)-\nabla_xH(x^{k-1},y^{k-1})+c_{k-1}(x^{k-1}-x^k)\in \partial F_\epsilon(x^k).$$
Finally, let us bound $\|w^k\|_2$ as follows
\begin{align}
\|w^k\|_2&=\|\nabla_xH(x^k,y^k)-\nabla_xH(x^{k-1},y^{k-1})+c_{k-1}(x^{k-1}-x^k)\|_2 \nonumber\\
& \leq c_{k-1}\|x^{k-1}-x^k\|_2+\|\nabla s(x^k)+\sum_{i=1}^m\nabla p_i(x^k)-\nabla s(x^{k-1})-\sum_{i=1}^m\nabla p_i(x^{k-1})\|_2 \nonumber\\
&\leq c_{k-1}\|x^{k-1}-x^k\|_2+\|\nabla s(x^k)-\nabla s(x^{k-1})\|_2+\sum_{i=1}^m\|\nabla p_i(x^k)-\nabla p_i(x^{k-1})\|_2 \nonumber\\
&\leq (c_{k-1}+L_s+ \sum_{i=1}^m L_i^\tau )\|x^{k-1}-x^k\|_2.
\end{align}
Since $c_{k-1}=\gamma L(\tau, y^{k-1})=\gamma (L_s+\|L_p^\tau\|_1\|y^{k-1}\|_\infty)\leq \gamma(L_s+\frac{\|L_p^\tau\|_1}{2\epsilon})$ and $\sum_{i=1}^m L_i^\tau=\|L_p^\tau\|_1$, we have that $$\|w^k\|_2\leq \left((\gamma+1)L_s+(\frac{\gamma}{2\epsilon}+1)\|L_p^\tau\|_1\right)\|x^{k-1}-x^k\|_2=\rho_2\|x^{k-1}-x^k\|_2.$$
This completes the proof.
\end{proof}

In the second one, we establish some results about the limit points of the sequence generated by PL-IRLS. Thereby, we define that $w(x^0):=\{u\in\RR^n: \exists~\textrm{an increasing sequence of integers}~\{k_j\}_{j\in\mathbb{N}}~\textrm{such that} ~x^{k_j}\rightarrow u ~\textrm{as}~ j\rightarrow \infty\}$, where $x^0\in\RR^n$ is an arbitrary starting point.
\begin{lemma}[Properties of the limit point set]\label{PLP}
 Let $\{x^k\}_{k\in\mathbb{N}}$ be a sequence generated by PL-IRLS and assume that $\inf F_\epsilon>-\infty$ and there exists a constant $\tau$ big enough such that $x^k\in \mathcal{B}(\tau), k=0, 1, \cdots.$ Then, the followings hold

 \begin{enumerate}
   \item[(i)] $\emptyset\neq w(x^0)\subset\textrm{crit}(F)$.
   \item[(ii)] $\lim_{k\rightarrow \infty} \textrm{dist}(x^k, w(x^0))=0$.
   \item[(iii)] $w(x^0)$ is a nonempty, compact, and connected set.
   \item[(iv)] $F_\epsilon(x)$ is finite and constant on $w(x^0)$.
 \end{enumerate}
\end{lemma}
\begin{proof}
\textit{(i)}
Let $x^*$ be a limit point of $\{x^k\}_{k\in\mathbb{N}}$. This means that there is a subsequence $\{x^{k_j}\}_{i\in\mathbb{N}}$ such that $x^{k_j}\rightarrow x^*$ as $j\rightarrow \infty$. Since $f(x)$ is lower semicontinuous, we obtain that
\begin{equation}\label{inf}
 \lim_{j\rightarrow \infty}\inf f(x^{k_j})\geq f(x^*).
\end{equation}
Recall that
\begin{equation}
x^{k+1}\in\arg\min_x\left( f(x)+\langle x-x^k,\nabla_xH(x^k, y^k)\rangle + \frac{c_k}{2}\|x-x^k\|_2^2\right).
\end{equation}
Thus, letting $x=x^*$ in the above, we obtain that
\begin{align}
&f(x^{k+1})+\langle x^{k+1}-x^k,\nabla_xH(x^k, y^k)\rangle + \frac{c_k}{2}\|x^{k+1}-x^k\|_2^2 \nonumber\\
 \leq &f(x^*)+\langle x^*-x^k,\nabla_xH(x^k, y^k)\rangle + \frac{c_k}{2}\|x^*-x^k\|_2^2.
\end{align}
Choosing $k=k_j-1$ above and letting $j$ tend to $\infty$, we have that
\begin{align}
 \lim_{i\rightarrow \infty}\sup f(x^{k_j}) & \leq \lim_{i\rightarrow \infty}\sup(\langle x^*-x^{k_j-1},\nabla_xH(x^{k_j-1}, y^{k_j-1})\rangle \nonumber\\
&+ \frac{c_{k_j-1}}{2}\|x^*-x^{k_j-1}\|_2^2+f(x^*)).
\end{align}
Since $\|x^{k_j}-x^{k_j-1}\|_2\rightarrow 0$ and $x^{k_j}\rightarrow x^*$ as $j\rightarrow \infty$ , we get that $x^{k_j-1}\rightarrow x^*$ as $j\rightarrow \infty$. Thus,
\begin{equation}\label{sup}
 \lim_{j\rightarrow \infty}\sup f(x^{k_j})\leq f(x^*).
\end{equation}
Combining \eqref{inf} and \eqref{sup} yields $\lim_{j\rightarrow\infty}f(x^{k_j})=f(x^*)$. Furthermore, we have that
\begin{align}
\lim_{j\rightarrow\infty}F_\epsilon(x^{k_j})  =& \lim_{j\rightarrow\infty}f(x^{k_j})+s(x^{k_j})+\sum_{i=1}^m\sqrt{\|B_ix^{k_j}-c_i\|_2^2+\epsilon^2} \nonumber\\
=&f(x^*)+s(x^*)+\sum_{i=1}^m\sqrt{\|B_ix^*-c_i\|_2^2+\epsilon^2}=F_\epsilon(x^*).
\end{align}
 By Lemma \ref{PIS}, we have that $w^{k_j}\in \partial F_\epsilon(x^k)$ and $w^{k_j}\rightarrow 0$ as $j\rightarrow \infty$. Together with that $\lim_{j\rightarrow\infty}F_\epsilon(x^{k_j})=F_\epsilon(x^*)$, the closedness property of $\partial F_\epsilon(x)$ implies that $0\in \partial F_\epsilon(x^*)$. This proves that $x^*\in \textrm{crit} F_\epsilon$.

\textit{ (ii)} and \textit{(iii) }follows from the fact that $\lim_{k\rightarrow  \infty}(x^{k+1}-x^k)=0$ proved in Lemma \ref{PIS} and Remark 5 in \cite{bolte2013proximal}.

 \textit{(iv)} Since the sequence $\{F_\epsilon(x^k)\}_{k\in\mathbb{N}}$  is nonincreasing and has a finite lower bound $\inf F_\epsilon$ , it must converge to a point, denoted by $c$ which is a finite constant. Take $x^*\in w(x^0)$. Then there is a subsequence $\{x^{k_j}\}_{i\in\mathbb{N}}$ such that $x^{k_j}\rightarrow x^*$ as $j\rightarrow \infty$. On one hand, it holds that  $c=\lim_{j\rightarrow \infty}F_\epsilon(x^{k_j})$ since $\{F_\epsilon(x^k)\}_{k\in\mathbb{N}}$ nonincreasely converges to $c$; On the other hand, we have shown that  $\lim_{j\rightarrow\infty}F_\epsilon(x^{k_j})=F_\epsilon(x^*)$. So $F_\epsilon(x^*)=c$. This completes the proof.
\end{proof}

\subsection{Convergence to a critical point}
Now, we are in a position to prove the main result.
\begin{theorem}[Main Result]\label{main1}
Suppose that $F_\epsilon(x)$ is a KL function with $\inf F_\epsilon>-\infty$. Let $\{x^k\}_{i\in\mathbb{N}}$ be a sequence generated by PL-IRLS and assume that there exists a constant $\tau$ big enough such that $x^k\in \mathcal{B}(\tau)$, for all $k\geq 0.$
\begin{enumerate}
  \item[(i)] The sequence $\{x^k\}_{i\in\mathbb{N}}$ has finite length, that is,
\begin{equation}
 \sum_{k=1}^\infty \|x^{k+1}-x^k\|_2<\infty.
\end{equation}
  \item[(ii)] The sequence $\{x^k\}_{i\in\mathbb{N}}$ converges to a critical point $x^*$ of $F_\epsilon$.
\end{enumerate}
\end{theorem}
With Lemmas \ref{UKL}, \ref{PIS}, and \ref{PLP} at hand, this theorem can be proved in a same way as in the proof of Theorem 1 in \cite{bolte2013proximal}. For completeness, we provide a short proof here.

\begin{proof}
\textit{(i)} The upcoming arguments heavily rely on Lemma \ref{UKL} with $\Omega:=w(x^0), \sigma:= F$. We begin with any point $\bar{u}\in w(x^0)$. Then, there exists an increasing sequences of integers $\{k_j\}_{j\in\mathbb{N}}$ such that $x^{k_j}\rightarrow \bar{u}$ as $j\rightarrow \infty$. Repeating the arguments in the proof of Lemma \ref{PLP} \textit{(iv)}, we get that
\begin{equation}
 \lim_{j\rightarrow\infty}F_\epsilon(x^{k_j})= \lim_{k\rightarrow\infty}F_\epsilon(x^k)= F_\epsilon(\bar{u}).
\end{equation}
Since $\{F_\epsilon(x^k)\}$ is nonincreasing and has a finite lower bound, if there exists an integer $\bar{k}$ such that $F_\epsilon(x^{\bar{k}})=F_\epsilon(\bar{u})$, then $F_\epsilon(x^k)\equiv F_\epsilon(\bar{u})$ for $k\geq \bar{k}$ which implies that $x^k\equiv x^{\bar{k}}$ for $k\geq \bar{k}$ from Lemma \ref{PIS} \textit{(i)}. In this case, the theorem holds obviously. For other cases, we assume that $F_\epsilon(x^k)>F_\epsilon(\bar{u})$, for all $k>0$. Since $\lim_{k\rightarrow\infty}F_\epsilon(x^k)= F_\epsilon(\bar{u})$, for any $\eta>0$ there must exist an integer $\hat{k}>0$ such that $F_\epsilon(x^k)<F_\epsilon(\bar{u})+\eta$ for all $k>\hat{k}$. Similarly, $\lim_{k\rightarrow \infty} \textrm{dist}(x^k, w(x^0))=0$ implies for any $\zeta>0$ there must exist an integer $\widetilde{k}>0$ such that $\textrm{dist}(x^k, w(x^0))<\zeta$ for all $k>\widetilde{k}$. Based on the discussion above, we obtain that for all $k>l:=\max\{\hat{k},\widetilde{k}\}$,
 \begin{equation}
 x^k\in \{u\in\RR^n:\textrm{dist}(u,\Omega)<\zeta\}\bigcap [F_\epsilon(\bar{u})<F_\epsilon(u)<F_\epsilon(\bar{u})+\eta].
\end{equation}
Thus, applying Lemma \ref{UKL} yields that for all $k>l$,
 \begin{equation}
  \varphi^{'}(F_\epsilon(x^k)-F_\epsilon(\bar{u}))\textrm{dist}(0,\partial F_\epsilon(x^k))\geq 1.
\end{equation}
By the definition of $\textrm{dist}(\cdot,\cdot)$ and $w^k\in \partial F_\epsilon(x^k)$ and Lemma \ref{PIS} \textit{(iii)}, we get that
\begin{equation}
 \textrm{dist}(0,\partial F_\epsilon(x^k))\leq \|w^k\|_2\leq \rho_2\|x^k-x^{k-1}\|_2.
\end{equation}
Hence,
\begin{equation}\label{bound1}
 \varphi^{'}(F_\epsilon(x^k)-F_\epsilon(\bar{u}))\geq \rho^{-1}_2\|x^k-x^{k-1}\|_2^{-1}.
\end{equation}
By the concavity of $\varphi$ and \eqref{bound1} and Lemma \ref{PLP} \textit{(i)}, we derive that
\begin{align}
&\varphi(F_\epsilon(x^k)-F_\epsilon(\bar{u}))-\varphi(F_\epsilon(x^{k+1})-F_\epsilon(\bar{u}))\nonumber\\
\geq & \varphi^{'}(F_\epsilon(x^k)-F_\epsilon(\bar{u}))(F_\epsilon(x^k)-F_\epsilon(x^{k+1}))\nonumber\\
\geq &\frac{F_\epsilon(x^k)-F_\epsilon(x^{k+1})}{\rho_2\|x^k-x^{k-1}\|_2}\geq \frac{\rho_1\|x^{k+1}-x^k\|_2^2}{\rho_2\|x^k-x^{k-1}\|_2}.
\end{align}
Define $\Delta_{s,t}:=\varphi(F_\epsilon(x^s)-F_\epsilon(\bar{u}))-\varphi(F_\epsilon(x^t)-F_\epsilon(\bar{u}))$ and $c:=\frac{\rho_2}{\rho_1}$. We obtain that
\begin{align}
\|x^{k+1}-x^k\|_2^2 \leq &c \cdot\Delta_{k,k+1}\|x^k-x^{k-1}\|_2\nonumber\\
\leq & \left(\frac{\|x^k-x^{k-1}\|_2+c \Delta_{k,k+1}}{2}\right)^2
\end{align}
i.e., $2\|x^{k+1}-x^k\|_2\leq \|x^k-x^{k-1}\|_2+c \cdot\Delta_{k,k+1}$ for all $k>l$. Summing up it for $i=l+1, \cdots, k$, we get that
\begin{align}
2\sum_{i=l+1}^k \|x^{i+1}-x^i\|_2 & \leq \sum_{i=l+1}^k\|x^i-x^{i-1}\|_2+c \sum_{i=l+1}^k \Delta_{i,i+1} \nonumber\\
 & \leq \sum_{i=l+1}^k\|x^{i+1}-x^i\|_2+ \|x^{l+1}-x^l\|_2+c \sum_{i=l+1}^k \Delta_{i,i+1} \nonumber\\
 & = \sum_{i=l+1}^k\|x^{i+1}-x^i\|_2+ \|x^{l+1}-x^l\|_2+c\cdot \Delta_{l+1,k+1}.
\end{align}
Note that $\varphi\geq 0$, it thus holds for any $k>l$,
$$\sum_{i=l+1}^k \|x^{i+1}-x^i\|_2\leq \|x^{l+1}-x^l\|_2+c\cdot \varphi(F_\epsilon(x^{l+1})-F_\epsilon(\bar{u})).$$
It implies that the sequence $\{x^k\}_{i\in\mathbb{N}}$ has finite length.

\textit{(ii)} $\sum_{k=1}^\infty \|x^{k+1}-x^k\|_2<\infty$ implies that $\{x^k\}$ is a Cauchy sequence and hence it is a convergent sequence. By Lemma \ref{PLP} (i), its limit point, denoted by $x^*$, belongs to $\textrm{crit}(F_\epsilon)$. This completes the proof.
\end{proof}

\section{Extension}
 Our method can be extended to solving the following more general nonconvex and nonsmooth problems:
\begin{equation*}
(GM)\quad \quad \Min f(x)+s(x)+\sum_{i=1}^m\|B_ix-c_i\|_2^\nu,
\end{equation*}
where the setting of $f(x), s(x), B_i, c_i, m$ is the same as that in (M), and $0<\nu\leq 1$ is a new generalized parameter. We can take the following problem
 \begin{equation*}
(GM_\epsilon)\quad \quad \Min F_{\epsilon,\nu}(x):=f(x)+s(x)+\sum_{i=1}^m (\|B_ix-c_i\|_2^2+\epsilon^2)^{\frac{\nu}{2}}
\end{equation*}
as a smoothed approximation and
\begin{equation*}
(GAMs)\quad\quad \Min \Psi(x, y)=f(x)+\underbrace{s(x)+\sum_{i=1}^m(\|B_ix-c_i\|_2^2+\epsilon^2)y_i}_{H(x,y)}+\underbrace{\sum_{i=1}^m \frac{\kappa}{y_i^\theta}+\delta(y, \overline{\Lambda})}_{g(y)}
\end{equation*}
as an auxiliary problem, where
\begin{equation}\label{para}\theta=\frac{\nu}{2-\nu},~~\overline{\Lambda}=(0, \frac{\nu}{2\epsilon^{2-\nu}}]^m, \kappa=\left(\frac{\nu}{2}\right)^\frac{2}{2-\nu}.
 \end{equation}
 It is easy to see that when $\nu=1$, we return to the problems (M), ($M_\epsilon$), and (AMs) respectively. PL-IRLS for the general problem ($GM_\epsilon$) is
 \begin{subequations}\label{ex}
\begin{equation}\label{ex1}
x^{k+1}\in\prox_{c_k}^f(x^k-\frac{1}{c_k}\nabla_xH(x^k,y^k))
\end{equation}
\begin{equation}\label{ex2}
y^{k+1}_i=\frac{\nu}{2}(\|B_ix^{k+1}-c_i
\|_2^2+\epsilon^2)^{\frac{\nu-2}{2}}, ~~i=1, 2, \cdots, m.
\end{equation}
\end{subequations}
Its globally convergence to a critical point can be proved in a similar way as before.

\begin{theorem}\label{main2}
Suppose that $F_{\epsilon,\nu}(x)$ is a KL function with $\inf F_{\epsilon,\nu}>-\infty$ . Let $\{x^k\}_{i\in\mathbb{N}}$ be a sequence generated by \eqref{ex1} and \eqref{ex2} and assume that there exists a constant $\tau$ big enough such that $x^k\in \mathcal{B}(\tau)$, for all $k\geq 0.$
\begin{enumerate}
  \item[(i)] The sequence $\{x^k\}_{i\in\mathbb{N}}$ has finite length, that is,
\begin{equation}
 \sum_{k=1}^\infty \|x^{k+1}-x^k\|_2<\infty.
\end{equation}
  \item[(ii)] The sequence $\{x^k\}_{i\in\mathbb{N}}$ converges to a critical point $x^*$ of $F_{\epsilon,\nu}(x)$.
\end{enumerate}
\end{theorem}

Our method may be extended to solve the following nonconvex and nonsmooth matrix-value functions minimization problem:
\begin{equation}
(MM)\quad \quad \Min f(X)+s(X)+tr[(XX^T)^{\frac{1}{2}}].
\end{equation}
Similarly, we can consider its smoothed approximation
\begin{equation}
(MM_\epsilon)\quad \quad \Min f(X)+s(X)+tr[(XX^T+\epsilon I)^{\frac{1}{2}}]
\end{equation}
and the corresponding auxiliary problem
\begin{equation}
(MAMs)\quad \quad \Min \Psi(X,Y):=f(X)+s(X)+tr[(XX^T+\epsilon I)Y + Y^{-1}]+\delta(Y, \mathcal{K}),
\end{equation}
where $\mathcal{K}$ is some positive-definite matrices cone. Note that $\partial tr(Y^{-1})=-(Y^{-2})^T$. If we fix $X$, then minimizing the objective $\Psi(X,Y)$ with respective to $Y$, we get a minimizer $Y=(XX^T+\epsilon I)^{-\frac{1}{2}}$. All of these observations make us believe that PL-IRLS can be extended to solve the matrix-value functions minimizations above. We leave it as a future work.

Finally, our method can also be extended to minimize objective function with two or more blocks of variables. For illustrating, we present two examples from low-rank and sparse matrices recovery.
\begin{equation}\label{newexample1}
 \Min_{X,Y\in \RR^{n\times n}} \|X\|_*+\|Y\|_1+\|\mathcal{A}(X+Y)-b\|_1
\end{equation}
\begin{equation}\label{newexample2}
 \Min_{U,V\in \RR^{n\times r}} \|\mathcal{A}(UV^T)-b\|_1,
\end{equation}
where $\mathcal{A}:\RR^{n\times n}\rightarrow \RR^m$ is a linear operator, $b\in\RR^m$ is an observed vector. The first example is referred to as sparse and low-rank matrices decomposition from observed data with sparse noise, and it is a convex programming. The second example is referred to as low-rank matrices decomposition from observed data with sparse noise, and it is a nonconvex programming.  To solve problems \eqref{newexample1} and \eqref{newexample2} by PL-IRLS, the main idea behind of our method is as same as before; that consists of two steps:
\begin{enumerate}
  \item Smooth the nondifferentiable and coupled term.
  \item Introduce a new variable to equivalently get an auxiliary problem.
\end{enumerate}
Take \eqref{newexample1} as an example. We first write down its smoothed version:
\begin{equation}
 \Min_{X,Y\in \RR^{n\times n}} \|X\|_*+\|Y\|_1+\sum_{i=1}^m\sqrt{(\mathcal{A}(X+Y)_i-b_i)^2+\epsilon^2}.
\end{equation}
And then we introduce a new vector $z$ and get
\begin{equation}
 \Min_{X,Y\in \RR^{n\times n}, z\in\RR^m} \|X\|_*+\|Y\|_1+\underbrace{\sum_{i=1}^m\left(((\mathcal{A}(X+Y)_i-b_i)^2+\epsilon^2)z_i+\frac{1}{4z_i}
 +\delta(z_i,(0,\frac{1}{2\epsilon}])\right)}_{H(X,Y,z)}.
\end{equation}
Applying PL-IRLS, we suggest the following scheme for solving \eqref{newexample1}.
\begin{subequations}
\begin{equation}\label{xkn}
X^{k+1}\in\prox_{c_k}^{\|\cdot\|_*}(X^k-\frac{1}{c_k}\nabla_x H(X^k,Y^k,z^k))
\end{equation}
\begin{equation}\label{ykn}
Y^{k+1}\in\prox_{d_k}^{\|\cdot\|_1}(Y^k-\frac{1}{d_k}\nabla_y H(X^{k+1},Y^k,z^k))
\end{equation}
\begin{equation}\label{zkn}
z_i^{k+1}\in\frac{1}{2\sqrt{(\mathcal{A}(X^{k+1}+Y^{k+1})_i-b_i)^2+\epsilon^2}}, i=1,\cdots, m.
\end{equation}
\end{subequations}
where $c_k, d_k$ are step parameters. The proximal maps of functions $\|\cdot\|_*$ and $\|\cdot\|_1$ have direct computation formulations. In fact, the proximal map of $\|\cdot\|_1$ is the soft-thresholding operator and the proximal map of $\|\cdot\|_*$  is the singular value thresholding operator \cite{cai2010singular}.  The global convergence to a critical of the objective function of \eqref{newexample1} then can be proved.

\section{Application}
From the convergence analysis before, two conditions are required to check before applying PL-IRLS to solve certain problems: The first one is whether the objective function $F_\epsilon(x)$ or $F_{\epsilon,\nu}(x)$ is a KL function; The second one is whether the proximal map of $f(x)$ can be easily computed.   In what follows, we solve three nonconvex examples appeared in signal/image processing to show how PL-IRLS can be applied to produce globally convergence algorithms. Convex cases are relatively easy.

\subsection{Nonconvex sparse least square problem}
We are interested in solving the following nonconvex unconstrained sparse least square problem
\begin{equation}
 \Min \frac{\lambda}{2}\|A x-b\|_2^2 +\|x\|_\nu^\nu,
\end{equation}
where $0< \nu\leq 1$. To apply PL-IRLS to this problem, we first need to write down its smooth approximation problem
\begin{equation}\label{unconspar2}
 \Min \frac{\lambda}{2}\|A x-b\|_2^2 +\sum_{i=1}^n(x_i^2+\epsilon^2)^{\frac{\nu}{2}}
\end{equation}
and its auxiliary problem
\begin{equation}
 \Min \Psi(x, y)=\underbrace{\frac{\lambda}{2}\|A x-b\|_2^2 +\sum_{i=1}^n(x_i^2+\epsilon^2)y_i}_{H(x,y)}+\underbrace{\sum_{i=1}^m \frac{\kappa}{y_i^\theta}+\delta(y, \overline{\Lambda})}_{g(y)},
\end{equation}
where the parameters are set as that in \eqref{para}.  In this problem, $f(x)$ disappears, so we do not need to concern the computation of proximal maps but need to check whether the objective function in \eqref{unconspar2} is a KL function. To do this, note that the function $\frac{\lambda}{2}\|A x-b\|_2^2$ is polynomial and hence is a KL function \cite{attouch2013convergence}, we only need to check $\sum_{i=1}^n(x_i^2+\epsilon^2)^{\frac{\nu}{2}}$.
\begin{lemma}\label{KLf}
Define $\|x\|_{\nu,\epsilon}:=\sum_{i=1}^n(x_i^2+\epsilon^2)^{\frac{\nu}{2}}$ with $\epsilon >0, 0<\nu\leq 1$. Then  $\|x\|_{\nu,\epsilon}$ is a KL function when $\nu$ is rational.
\end{lemma}
\begin{proof}
Let $\nu=\frac{p_1}{p_2}$ where $p_1, p_2$ are positive integers. Since the composition of semi-algebraic functions is also a semi-algebraic function, it suffices to prove that $u\rightarrow (u^2+\epsilon^2)^{\frac{p_1}{p_2}}$ is semi-algebraic. Its graph $\RR^2$ can be written as
$$\{(u, t)\in \RR^2: t=(u^2+\epsilon^2)^{\frac{p_1}{p_2}}\}=\{(u, t)\in \RR^2: t^{p_2}-(u^2+\epsilon^2)^{p_1}=0\},$$
which is obviously a semi-algebraic set by definition. So $u\rightarrow (u^2+\epsilon^2)^{\frac{p_1}{p_2}}$ is a semi-algebraic function. Since $\textrm{dom}\partial \|x\|_{\nu,\epsilon}=\textrm{dom}\|x\|_{\nu,\epsilon}$, by Lemma \ref{semialge} we can conclude that $\|x\|_{\nu,\epsilon}$ is a KL function when $\nu$ is rational.
\end{proof}
We have known that finite sum of KL functions is also a KL function. Therefore, the objective function in \eqref{unconspar2} is a KL function when $\nu$ is rational, and hence PL-IRLS can be safely applied to solve problem \eqref{unconspar2}. Let $Y=\textrm{diag}(y_1,\cdots,y_n)$; then by simple calculation, we get that $\nabla_x H(x,y)=\lambda A^T(Ax-b)+2Yx$. Applying \eqref{ex} to problem \eqref{unconspar2}, we obtain that
 \begin{subequations}\label{aex}
\begin{equation}\label{aex1}
x^{k+1}=x^k-\frac{1}{c_k}(\lambda A^T(Ax^k-b)+2Y^kx^k)
\end{equation}
\begin{equation}\label{aex2}
y^{k+1}_i=\frac{\nu}{2}((x_i^{k+1})^2+\epsilon^2)^{\frac{\nu-2}{2}}, ~~i=1, 2, \cdots, n,
\end{equation}
\end{subequations}
where $Y^k=\textrm{diag}(y_1^k,\cdots,y_n^k)$. The global convergence of $\{x^k\}$ generated by the iterative algorithm above to a critical point of  the objective function in \eqref{unconspar2} can be guaranteed by Theorem \ref{main2}. At the end, we would like to mention a similar iteratively reweighted algorithm for solving \eqref{unconspar2} in \cite{lai2011unconstrained}. That algorithm, which we will call IR algorithm, can be described by two steps

  \begin{enumerate}
    \item[(a)] Obtain  $x^{k+1}$ by solving $\nabla_x H(x,y^k)=\lambda A^T(Ax-b)+2Y^kx=0$.

    \item[(b)] Compute $y^{k+1}$ by utilizing \eqref{aex2}.
  \end{enumerate}

  The authors in \cite{lai2011unconstrained} proved that under certain conditions, the accumulation points of the sequence generated by the IR algorithm can be stationary points of the objective function in \eqref{unconspar2}. The merit of the IR algorithm is that it can be used for sparse recovery. Its potential drawback is the difficulty of solving the linear system of $\lambda A^T(Ax-b)+2Y^kx=0$. In addition, its global convergence needs to be proved.

\subsection{Nonconvex sparse $\ell_1$-norm regression}
We are interested in the following unconstrained sparse $\ell_1$-norm regression problem
\begin{equation}
 \Min \lambda\|x\|_0 +\|Ax-b\|_1.
\end{equation}
To apply PL-IRLS to this problem, we first need to write down its smooth approximation problem
\begin{equation}\label{unconspar}
 \Min \lambda\|x\|_0 +\sum_{i=1}^m\sqrt{(Ax-b)_i^2+\epsilon^2}
\end{equation}
and its auxiliary problem
\begin{equation}
 \Min \Psi(x, y)=\underbrace{\lambda\|x\|_0 }_{ f(x)} +\underbrace{\sum_{i=1}^m((Ax-b)_i^2+\epsilon^2)y_i}_{H(x,y)}+\underbrace{\sum_{i=1}^m\frac{1}{4y_i}+\delta(y, \Lambda)}_{g(y)}.
\end{equation}
Let $Y=\textrm{diag}(y_1,\cdots,y_n)$; then by simple calculation, we get that $\nabla_xH(x,y)=2A^TYAx-2A^TYb$.
Second, $\lambda\|x\|_0$ is a KL function (see \cite{bolte2013proximal}) and $\sum_{i=1}^m\sqrt{(Ax-b)_i^2+\epsilon^2}$ is also KL by Lemma \ref{KLf}, and so is their sum. Third, the proximal map of $\lambda\|x\|_0$ can be easily computed. In fact, when $n=1$, the counting norm is denoted by $|\cdot|_0$ and the authors in \cite{attouch2013convergence} establishes that
\begin{eqnarray}
\prox_c^{\lambda |\cdot|_0}(u)=\left\{\begin{array}{ll}
u &\textrm{if} ~|u|>\sqrt{2\lambda/c}\\
\{0,u\}&\textrm{if} ~|u|=\sqrt{2\lambda/c}\\
0&\textrm{otherwise},
\end{array} \right.
\end{eqnarray}
and for $y\in\RR^n$,
$$\prox_c^{\lambda \|\cdot\|_0}(y)=(\prox_c^{\lambda |\cdot|_0}(y_1),\cdots, \prox_c^{\lambda |\cdot|_0}(y_n))^T.$$
Now, applying PL-IRLS  to  \eqref{unconspar}, we obtain that
 \begin{subequations}
\begin{equation}\label{app1}
x^{k+1}\in \prox_{c_k}^{\lambda \|\cdot\|_0}(x^k-\frac{2}{c_k}A^TY^k(Ax^k-b) )
\end{equation}
\begin{equation}\label{app2}
y^{k+1}_i=\frac{1}{2\sqrt{(Ax^{k+1}-b)_i^2+\epsilon^2}}
, ~~i=1, 2, \cdots, n,\end{equation}
\end{subequations}
where $Y^k=\textrm{diag}(y_1^k,\cdots,y_n^k)$. By Theorem \ref{main1}, the algorithm above is guaranteed to globally converge to a critical point of the objective function in \eqref{unconspar}.

\subsection{Nonconvex low-rank matrix recovery}
In low-rank matrix recovery, one may be interested in the following problem
\begin{equation}\label{rpca11}
\Min \lambda \cdot\textrm{rank}(X)+\|D-X\|_1.
\end{equation}
We suggest PL-IRLS solve it based on the fact that $\lambda \cdot\textrm{rank}(X)$ is a KL function \cite{bolte2013proximal} and the following observation:
\begin{lemma}\label{rank}
Let matrix $Y\in \RR^{m\times n}$ have singular value decomposition $Y=U\Sigma V^T$ with $U\in\RR^{m\times m}, V\in\RR^{n\times n}$ are orthogonal matrices and $\Sigma=[\sigma_{ij}]\in\RR^{m\times n}$ having $\sigma_{ij=0}$ for all $i\neq j$, and $\sigma_{11}\geq \sigma_{22}\geq \cdots\geq \sigma_{kk}>\sigma_{k+1,k+1}=\cdots=\sigma_{q,q}=0$, where $k=\textrm{rank}(Y)$ and $q=\min\{n,m\}$. Then, $U\hat{Z}V^T\in \prox^{rank(\cdot)}_c(Y)$ for each $\hat{Z}$ with $\hat{Z}_{ii}\in\prox^{|\cdot|_{0}}_{c}(\sigma_{ii})$, $i=1,\cdots, q$ and other entries equal to zero.
\end{lemma}
\begin{proof}
We begin with the definition of $\prox^{rank(\cdot)}_c(Y)$,
\begin{subequations}
\begin{eqnarray}
    \prox^{\textrm{rank}(\cdot)}_c(Y)&=&\textrm{arg}\min_{X}\{rank(X)+\frac{c}{2}\parallel X-Y\parallel_{F}^{2}\} \nonumber \\
    &=&\textrm{arg}\min_{X}\{rank(X)+\frac{c}{2}\parallel X-U\Sigma V^T \parallel_{F}^{2}\} \nonumber\\
    &=&\textrm{arg}\min_{X}\{rank(U^TXV)+\frac{\lambda}{2}\parallel U^TXV-\Sigma \parallel_{F}^{2}\} \nonumber\\ &=&\arg\min_{Z}\{rank(Z)+\frac{c}{2}\parallel Z-\Sigma \parallel_{F}^{2}\}, \nonumber
\end{eqnarray}
\end{subequations}
where $Z=U^TXV$.
Define $G(Z):=rank(Z)+\frac{c}{2}\parallel Z-\Sigma \parallel_{F}^{2}$ and $\prod:=\{Z\in \RR^{m\times n}: Z_{ij}=0, i\neq j\}$. Let $\bar{Z}\in \arg\min_{Z}G(Z)$ with $rank(\bar{Z})=r$.  Then, $$\bar{Z}\in\arg\min_{rank(Z)=r}\|Z-\Sigma\|_F^2.$$
By the Eckart-Young theorem \cite{horn1990matrix}, we have that $\Sigma_r\in\arg\min_{rank(Z)=r}\|Z-\Sigma\|_F^2$ where $\Sigma_r=[\pi_{ij}]$ has $\pi_{11}=\sigma_{11},  \pi_{22}=\sigma_{22},\cdots, \pi_{rr}=\sigma_{rr}$ and other entries equal to zero. It is easy to see that $G(\Sigma_r)=G(\bar{Z})$ and hence $\Sigma_r\in \arg\min_{Z}G(Z)$. Noting $\Sigma_r\in \prod$, we obtain that
$$\min_{Z\in \prod}G(Z)\leq G(\Sigma_r)=\min_ZG(Z).$$
On the other hand, it holds that $\min_ZG(Z)\leq \min_{Z\in \prod}G(Z)$. Therefore, $\min_ZG(Z)= \min_{Z\in \prod}G(Z)$ which implies that
\begin{equation}\label{subsolu}
\arg\min_{Z\in \prod}G(Z)\subseteq\arg\min_{Z}G(Z).
\end{equation}
Let $u=(Z_{11}, Z_{22}, \cdots, Z_{qq})^T$ and $v=(\sigma_{11}, \sigma_{22}, \cdots, \sigma_{qq})^T$. Then, $\arg\min_{Z\in \prod}G(Z)$ can be reduced to
$$\arg\min_{u\in\RR^q}\|u\|_0+\frac{c}{2}\|u-v\|_2^2=\prox^{\|\cdot\|_0}_c(v).$$
Thus, from the relationship \eqref{subsolu}, the conclusion follows.
\end{proof}

Consider the following smoothed approximation of \eqref{rpca11}
\begin{equation}\label{rpca12}
\Min \textrm{rank}(X)+\sum_{i,j}\frac{1}{\lambda}\sqrt{(D_{ij}-X_{ij})^2+\epsilon^2},
\end{equation}
and its auxiliary problem
\begin{equation}\label{rpca13}
\Min \textrm{rank}(X)+\sum_{i,j}\frac{1}{\lambda}[(D_{ij}-X_{ij})^2+\epsilon^2]Y_{ij}+\sum_{i,j}\left(\frac{1}{4Y_{ij}}+\delta(Y_{ij}, (0,\frac{\epsilon}{2}])\right).
\end{equation}
Applying PL-IRLS, we obtain that
 \begin{subequations}
\begin{equation}\label{app3}
X^{k+1}\in \prox_{c_k}^{rank(\cdot)}(X^k-\frac{2}{c_k}\nabla_X H(X^k,Y^k))
\end{equation}
\begin{equation}\label{app4}
Y^{k+1}_{ij}=\frac{1}{2\sqrt{(D_{ij}-X_{ij}^{k+1})^2+\epsilon^2}},
\end{equation}
\end{subequations}
where $(\nabla_X H(X^k,Y^k))_{ij}=\frac{2}{\lambda}Y_{ij}^k(X_{ij}^k-D_{ij})$. Let $Z^k=X^k-\frac{2}{c_k}\nabla_X H(X^k,Y^k)$; then $X^{k+1}\in \prox_{c_k}^{rank(\cdot)}(Z^k)$. Now, the main difficulty is how to compute the proximal map of $rank(\cdot)$ at $Z^k$. Lemma \ref{rank} gives a way to do this, so we omit the details here.

\section*{Acknowledgements}
The authors thank Prof.Wotao Yin (UCLA) for helpful comments. The work of H. Zhang is supported by Graduate School of NUDT under Funding of Innovation B110202, The work of T. Sun is supported by NSF Grants  No.61201328 and NNSF of Hunan Province(13JJ4011).  The work of L. Cheng is supported by NSF Grants No.61271014 and No.61072118, and NNSF of Hunan Province(13JJ2011), and Science Projection of NUDT (JC120201).


\end{document}